\documentclass[11pt]{amsart}
\usepackage{amssymb}
\usepackage{bm}

\usepackage{epsfig}
 \theoremstyle{plain}
\newtheorem{theorem}{Theorem}
\newtheorem{corollary}[theorem]{Corollary}

\newtheorem{lemma}[theorem]{Lemma}
\newtheorem{proposition}[theorem]{Proposition}
\newtheorem{example}{Example}
\theoremstyle{definition}
\newtheorem{definition}{Definition}
\theoremstyle{remark}

\numberwithin{equation}{section}
\numberwithin{theorem}{section}

\newcommand{\bT}{\begin{theorem}}
\newcommand{\eT}{\end{theorem}}
\newcommand{\bProp}{\begin{proposition}}
\newcommand{\eProp}{\end{proposition}}
\newcommand{\bE}{\begin{example}}
\newcommand{\eE}{\end{example}}
\newcommand{\bL}{\begin{lemma}}
\newcommand{\eL}{\end{lemma}}
\newcommand{\bP}{\begin{proof}}
\newcommand{\eP}{\end{proof}}
\newcommand{\bC}{\begin{corollary}}
\newcommand{\eC}{\end{corollary}}
\newcommand{\bD}{\begin{definition}}
\newcommand{\eD}{\end{definition}}
\newcommand{\be}{\begin{enumerate}}
\newcommand{\ee}{\end{enumerate}}
\newcommand{\beqa}{\begin{eqnarray*}}
\newcommand{\eeqa}{\end{eqnarray*}}
\newcommand{\beqaa}{\begin{eqnarray}}
\newcommand{\eeqaa}{\end{eqnarray}}
\newcommand{\ba}{\begin{array}}
\newcommand{\ea}{\end{array}}

\setbox0=\hbox{$+$}
\newdimen\plusheight
\plusheight=\ht0
\def\+{\;\lower\plusheight\hbox{$+$}\;}

\setbox0=\hbox{$-$}
\newdimen\minusheight
\minusheight=\ht0
\def\-{\;\lower\minusheight\hbox{$-$}\;}

\setbox0=\hbox{$\cdots$}
\newdimen\cdotsheight
\cdotsheight=\plusheight
\def\cds{\lower\cdotsheight\hbox{$\cdots$}}
\begin{document}
\title[Lambert series, vanishing coefficients and $m$-dissections ]
       {Some Observations on Lambert series, vanishing coefficients and dissections of infinite products and series   }
       {\allowdisplaybreaks
\author{James Mc Laughlin}
\address{Mathematics Department\\
 25 University Avenue\\
West Chester University, West Chester, PA 19383}
\email{jmclaughl@wcupa.edu}}


\date{\today}

\begin{abstract}
Andrews and Bressoud, Alladi and Gordon, and others, have proven, in a number of papers, that the coefficients in various arithmetic progressions in the series expansions of certain infinite $q$-products vanish. In the present paper it is shown that these results follow automatically (simply by specializing parameters) in an identity derived from a special case of Ramanujan's $_1\psi_1$ identity.

Likewise, a number of authors have proven results about the $m$-dissections of certain infinite $q$-products using various methods.  It is  shown that many of these $m$-dissections also follow automatically (again simply by specializing parameters) from this same identity alluded to above.

Two identities that mat be considered as extensions of two Identities of Ramanujan are also derived.

It is also shown how applying  similar ideas to certain other Lambert series gives rise to some rather curious $q$-series identities, such as, for any positive integer $m$,
\begin{multline*}
{\displaystyle
\frac{\left(q,q,a,\frac{q}{a},\frac{b q}{d}, \frac{dq}{b}, \frac{aq}{b d},
\frac{b d q}{a};q\right)_{\infty }}
{\left(b,\frac{q}{b},d,\frac{q}{d},\frac{a}{b},\frac{bq}{a},\frac{a}{d},\frac{dq}{a};q\right)_{\infty }}}\\
   =
\sum _{r=0}^{m-1} q^r
\frac{
\left(q^m,q^m,a q^{2 r},\frac{q^{m-2 r}}{a},\frac{b q^m}{d},\frac{d q^m}{b},
\frac{a q^m}{b d},\frac{b dq^m}{a};q^m\right){}_{\infty }}
{\left(b q^r,\frac{q^{m-r}}{b},d q^r,\frac{q^{m-r}}{d},\frac{a q^r}{b},\frac{b q^{m-r}}{a},\frac{a q^r}{d},
\frac{dq^{m-r}}{a};q^m\right){}_{\infty }}
\end{multline*}
 and
 \begin{equation*}
(aq;q)_{\infty}\sum_{n=1}^{\infty} \frac{n a^n q^{n}}{(q;q)_n}
= \sum_{r=1}^{m}(aq^{r};q^m)_{\infty}\sum_{n=1}^{\infty} \frac{na^n q^{n r}}{(q^m;q^m)_n}.
\end{equation*}
Applications to the Fine function $F(a,b;t)$ are also considered.
\end{abstract}

\maketitle

\section{Introduction}

The origins of this paper lie in the entirely elementary observation that if
\begin{equation}\label{lam1eqa}
f(a,b,c,q):=\sum_{n=-\infty}^{\infty}\frac{a b^n}{1-c q^n},
\end{equation}
and $m$ is any positive integer, then after replacing $n$ with $nm+r$ and rearranging into arithmetic progressions modulo $m$,
\begin{equation}\label{lam1eqb}
f(a,b,c,q)=\sum_{r=0}^{m-1}f(a b^r, b^m, c q^r, q^m).
\end{equation}
A similar statement can be made for a  function
\begin{equation}\label{lam2eqa}
f^{*}(a,b,c,q):=\sum_{n=0}^{\infty}\frac{a b^n}{1-c q^n}.
\end{equation}

The usefulness of this derives from the fact if $f(a,b,c,q)$ has some other representation as a basic hypergeometric series or infinite $q$-product, say $f(a,b,c,q) = g(a,b,c,q)$, then for any positive integer $m$,
\begin{equation}\label{lam1eqc}
g(a,b,c,q)=\sum_{r=0}^{m-1}g(a b^r, b^m, c q^r, q^m).
\end{equation}
A similar statement may be made about $f^{*}(a,b,c,q)$.

As will be shown, this will lead to proofs of various results on vanishing coefficients, $m$-dissections of various infinite products, and will also result in a number of new, somewhat unusual
$q$-series expansions involving an integer parameter $m$. In the case of the results on vanishing coefficients and dissection of infinite products, the results will follow immediately upon specialization of parameters in a general identity, in contrast to previous proofs of these results, which employed various manipulations of infinite series.

Some examples of existing results proved by the methods of the present paper are given below. The
first two use a general identity derived from Ramanujan's $_1\psi_1$ summation formula (valid for $|q|<1$, $|b/a|<|z|<1$),
\begin{equation}\label{ram1}
\sum_{n=-\infty}^{\infty} \frac{(a;q)_nz^n}{(b;q)_n}=\frac{(b/a,q,az,q/az;q)_{\infty}}{(q/a,b,z,b/az;q)_{\infty}},
\end{equation}

The first example is a general result on vanishing coefficients.
\begin{theorem}[Mc Laughlin, \cite{McL15}]
Let $k>1$, $m>1$ be positive integers. Let $r=sm+t$, for some integers $s$ and $t$, where $0\leq s<k$, $1\leq t <m$ and $r$ and $k$ are relatively prime. Let
\begin{equation}
\frac{(q^{r-tk},q^{mk-(r-tk)};q^{mk})_{\infty}}{(q^{r},q^{mk-r};q^{mk})_{\infty}}=:\sum_{n=0}^{\infty} c_nq^n,
\end{equation}
then $c_{kn-rs}$ is always zero.
\end{theorem}

The second is an example of the application of the general identity to $m$-dissections of infinite products, where the 3-dissection of the infinite product on the left below is given.
\begin{corollary}$($Lin, \cite[Theorem 3.2]{L13}$)$
If $|q|<1$, then
\begin{multline}
\frac{(q^{7},q^{5};q^{12})_{\infty}}
 {(q,q^{11};q^{12})_{\infty}}
 =\frac{(q^{6},q^{30};q^{36})_{\infty}}
 {(q^{12},q^{12},q^{24},q^{24};q^{36})_{\infty}} \times\\
 \left(
 \frac{(q^{9},q^{18},q^{18},q^{27};q^{36})_{\infty}}
 {(q^{3},q^{33};q^{36})_{\infty}}
 +
 q\frac{(q^{6},q^{15},q^{21},q^{30};q^{36})_{\infty}}
 {(q^{3},q^{33};q^{36})_{\infty}}
+
q^2(q^{18},q^{18};q^{36})_{\infty}
 \right).
\end{multline}
\end{corollary}

Iteration of the main identity leads to infinite product identities such as 
\begin{multline}
 \frac{(q^5,q^5;q^5)_{\infty}}
 {(q,q^4;q^5)_{\infty}}\\
  =
 \frac{1}
 {1-q^2}
 \prod_{k=1}^{\infty}
 \left(1+
 q^{2^{k-1}}
 \frac{(q^2,q^{2^k5-2}, q^{2^{k-1}7+2},q^{2^{k-1}3-2};q^{2^k5})_{\infty}}
 {(q^{2^{k-1}5+2},q^{2^{k-1}5-2},q^{2^k+2},q^{2^{k+2}-2};q^{2^k5})_{\infty}}
 \right).
 \end{multline}

After applying similar ideas to Bailey's $_6\psi_6$ summation formula, the following identity is a consequence.

\begin{theorem}
If $|q|<1$ and $m$ is a positive integer, then
\begin{multline}
{\displaystyle
\frac{\left(q,q,a,\frac{q}{a},\frac{b q}{d}, \frac{dq}{b}, \frac{aq}{b d},
\frac{b d q}{a};q\right)_{\infty }}
{\left(b,\frac{q}{b},d,\frac{q}{d},\frac{a}{b},\frac{bq}{a},\frac{a}{d},\frac{dq}{a};q\right)_{\infty }}}\\
   =
\sum _{r=0}^{m-1} q^r
\frac{
\left(q^m,q^m,a q^{2 r},\frac{q^{m-2 r}}{a},\frac{b q^m}{d},\frac{d q^m}{b},
\frac{a q^m}{b d},\frac{b dq^m}{a};q^m\right){}_{\infty }}
{\left(b q^r,\frac{q^{m-r}}{b},d q^r,\frac{q^{m-r}}{d},\frac{a q^r}{b},\frac{b q^{m-r}}{a},\frac{a q^r}{d},
\frac{dq^{m-r}}{a};q^m\right){}_{\infty }}.
\end{multline}
\end{theorem}

Finally, examples of applying similar ideas to other identities leads to the following somewhat unusual expansions involving the integer parameter $m$.

\begin{corollary}
Let $m$ be a positive integer.\\
(i) If $|q|<1$, then
\begin{multline}
\frac{1}{(aq;q)_{\infty}}\sum_{n=1}^{\infty} \frac{n(-a)^n q^{n(n+1)/2}}{(q;q)_n}\\
= \sum_{r=1}^{m}\frac{1}{(aq^{r};q^m)_{\infty}}\sum_{n=1}^{\infty} \frac{n(-a)^n q^{mn(n-1)/2+nr}}{(q^m;q^m)_n}.
\end{multline}
If $|q|$, $|aq|<1$, then
\begin{equation}
(aq;q)_{\infty}\sum_{n=1}^{\infty} \frac{n a^n q^{n}}{(q;q)_n}
= \sum_{r=1}^{m}(aq^{r};q^m)_{\infty}\sum_{n=1}^{\infty} \frac{na^n q^{n r}}{(q^m;q^m)_n}.
\end{equation}
\end{corollary}

\section{A special case Ramanujan's $_1\psi_1$ summation formula}

 The main tool used is an identity that follows from a special case Ramanujan's $_1\psi_1$ summation formula \eqref{ram1}, after specializing one of the parameters. Observe that the right side of \eqref{prop1} provides the $p$-dissection of the left side, where by ``$p$-dissection'' here we mean in terms of powers of $z$, rather than powers of $q$.

 \begin{proposition}\label{prop1}
 If $p$ is a positive integer and $|q|<1$, then
 \begin{equation}\label{midisszram}
 \frac{(q,q,az,q/(az);q)_{\infty}}{(a,q/a,z,q/z;q)_{\infty}}
 =
 \sum_{j=0}^{p-1}z^j
 \frac{(q^p,q^p,aq^jz^p,q^{p-j}/(az^p);q^p)_{\infty}}
 {(aq^j,q^{p-j}/a,z^p,q^p/z^p;q^p)_{\infty}}.
 \end{equation}
 \end{proposition}
 \begin{proof}
 Set $b=aq$ in \eqref{ram1} and divide both sides by $1-a$ to get
 {\allowdisplaybreaks
 \begin{align}\label{prop1eq1}
 \frac{(q,q,az,q/(az);q)_{\infty}}
 {(a,q/a,z,q/z;q)_{\infty}}
 &= \sum_{n=-\infty}^{\infty}\frac{z^n}{1-aq^n}\\
 &= \sum_{j=0}^{p-1}\sum_{n=-\infty}^{\infty}\frac{z^{np+j}}{1-aq^{np+j}}\notag \\
 &= \sum_{j=0}^{p-1}z^j\sum_{n=-\infty}^{\infty}\frac{(z^{p})^n}{1-(aq^j)(q^{p})^n}, \notag
 \end{align}}
 and \eqref{midisszram} for $|q|<|z|<1$ follows immediately upon employing the first equality in \eqref{prop1eq1}, with $q$ replaced with $q^p$, $a$ with $aq^r$, and $z$ with $z^p$ to each of the inner sums in the last expression. The full result follows by analytic continuation.
 \end{proof}
 The identity \eqref{midisszram} can be slightly rearranged and iterated to give the following identity.
 \begin{corollary}
 If $m$ and $p$ are positive integers, then
 \begin{multline}\label{midisszram2}
 \frac{(q,q,az,q/(az);q)_{\infty}}
 {(a,q/a,z,q/z;q)_{\infty}}
 =
 \frac{(q^{p^m},q^{p^m},az^{p^m},q^{p^m}/(az^{p^m});q^{p^m})_{\infty}}
 {(a,q^{p^m}/a,z^{p^m},q^{p^m}/z^{p^m};q^{p^m})_{\infty}}\\
 \times
 \prod_{k=1}^{m}
 \left(1+
 \sum_{j=1}^{p-1}z^{jp^{k-1}}
 \frac{(a,q^{p^k}/a,a q^{jp^{k-1}}z^{p^k},q^{p^k-jp^{k-1}}/(az^{p^k});q^{p^k})_{\infty}}
 {(aq^{jp^{k-1}},q^{p^k-jp^{k-1}}/a,az^{p^k},q^{p^k}/(az^{p^k});q^{p^k})_{\infty}}
 \right).
 \end{multline}
 \end{corollary}
 \begin{proof}
 Rewrite \eqref{midisszram} in the form
 \begin{multline}\label{midisszram3}
 \frac{(q,q,az,q/(az);q)_{\infty}}
 {(a,q/a,z,q/z;q)_{\infty}}\\
 =
 \frac{(q^p,q^p,az^p,q^{p}/(az^p);q^p)_{\infty}}
 {(a,q^{p}/a,z^p,q^p/z^p;q^p)_{\infty}}
 \sum_{j=0}^{p-1}z^j
 \frac{(a,q^p/a,aq^jz^p,q^{p-j}/(az^p);q^p)_{\infty}}
 {(aq^j,q^{p-j}/a,az^p,q^p/(az^p);q^p)_{\infty}}.
 \end{multline}
 Now iterate, noting that the infinite product before the sum on the right side of \eqref{midisszram3} has the same form as the infinite product on the left side, with the replacements $q\to q^p$ and $z\to z^p$. Note also that the $j=0$ term in the sum on the right simplifies to 1.
 \end{proof}
 The limit as $m\to \infty$ of the individual terms on the right side of \eqref{midisszram3}  will not exist for all values of $z$, but will for some values of $z$. We next give an example of an identity where these limits do exist,  derived by making the substitutions $q\to q^5$, $z\to q$, $a\to q^2$ in \eqref{midisszram3} and setting $p=2$.
 \begin{example}
 If $m$ is a positive integer, then
 \begin{multline}\label{midisszram4}
 \frac{(q^5,q^5;q^5)_{\infty}}
 {(q,q^4;q^5)_{\infty}}
 =
 \frac{(q^{2^m 5},q^{2^m 5},q^{2^m+2},q^{2^{m+2} -2};q^{2^m 5})_{\infty}}
 {(q^2,q^{2^m 5-2},q^{2^m},q^{2^{m+2}};q^{2^m 5})_{\infty}}\\
 \times
 \prod_{k=1}^{m}
 \left(1+
 q^{2^{k-1}}
 \frac{(q^2,q^{2^k5-2}, q^{2^{k-1}7+2},q^{2^{k-1}3-2};q^{2^k5})_{\infty}}
 {(q^{2^{k-1}5+2},q^{2^{k-1}5-2},q^{2^k+2},q^{2^{k+2}-2};q^{2^k5})_{\infty}}
 \right).
 \end{multline}
 Upon letting $m \to \infty$, it follows that
 \begin{multline}\label{midisszram5}
 \frac{(q^5,q^5;q^5)_{\infty}}
 {(q,q^4;q^5)_{\infty}}\\
  =
 \frac{1}
 {1-q^2}
 \prod_{k=1}^{\infty}
 \left(1+
 q^{2^{k-1}}
 \frac{(q^2,q^{2^k5-2}, q^{2^{k-1}7+2},q^{2^{k-1}3-2};q^{2^k5})_{\infty}}
 {(q^{2^{k-1}5+2},q^{2^{k-1}5-2},q^{2^k+2},q^{2^{k+2}-2};q^{2^k5})_{\infty}}
 \right).
 \end{multline}
 \end{example}

 \section{Extensions of two identities of Ramanujan}

 Before considering vanishing coefficients and $m$-dissections, we consider two other quite general identities, which may be regarded as
 extensions of two identities of Ramanujan, as will be shown below.
 \begin{corollary}\label{1psi1cR}
 Let $m$ be a positive integer, and let $\omega = \exp(2\pi i/m)$ be a primitive $m$ root of unity. If $a,z\not = 0$ and $|q|<1$, then
 \begin{equation}\label{1psi1cReq1}
 \sum_{j=0}^{m-1}
 \frac{\left(q,q,\omega^jaz,\displaystyle{\frac{q}{\omega^jaz}};q\right)_{\infty}}
 {\left(a,\displaystyle{\frac{q}{a}},\omega^j z,\displaystyle{\frac{q}{\omega^j z}};q\right)_{\infty}}
 =
 m\frac{\left(q^m,q^m,az^m,\displaystyle{\frac{q^m}{az^m}};q^m\right)_{\infty}}
 {\left(a,\displaystyle{\frac{q^m}{a}}, z^m,\displaystyle{\frac{q^m}{z^m}};q^m\right)_{\infty}},
 \end{equation}
 \begin{equation}\label{1psi1cReq2}
 \sum_{j=0}^{m-1}\omega^j
 \frac{\left(q,q,\omega^jaz,\displaystyle{\frac{q}{\omega^jaz}};q\right)_{\infty}}
 {\left(\omega^ja,\displaystyle{\frac{q}{\omega^j a}}, z,\displaystyle{\frac{q}{z}};q\right)_{\infty}}
 =
 m a^{m-1}\frac{\left(q^m,q^m,a^m z q^{m-1},\displaystyle{\frac{q}{a^m z}};q^m\right)_{\infty}}
 {\left(a^m,\displaystyle{\frac{q^m}{a^m}}, z q^{m-1},\displaystyle{\frac{q}{z}};q^m\right)_{\infty}}.
 \end{equation}
 \end{corollary}
\begin{proof}
For \eqref{1psi1cReq1}, take the first equality \eqref{prop1eq1} and for each $j$, $0\leq j\leq m-1$, replace $z$ with $\omega^j z$. Add the resulting $m$ identities together, so that the resulting left side is the left side of \eqref{1psi1cReq1}, and the right side becomes
\[
\sum_{n=-\infty}^{\infty}\frac{z^n}{1-aq^n}\sum_{j=0}^{m-1}(\omega^n)^j.
\]
That this equals the right side of \eqref{1psi1cReq1} follows from the fact that
\[
\sum_{j=0}^{m-1}(\omega^n)^j
=
\begin{cases}
m, & m|n,\\
0,&\text{otherwise}.
\end{cases}
\]
This proves \eqref{1psi1cReq1} for $|q|<|z|<1$, and the full result follows by analytic continuation.

For \eqref{1psi1cReq2}, similarly take the first equality \eqref{prop1eq1} and for each $j$, $0\leq j\leq m-1$, replace $a$ with $\omega^j a$. Multiply the $j$-th resulting identity by $\omega^j$ and add all $m$ identities together, so that the resulting left side is the left side of \eqref{1psi1cReq2}, and the right side becomes
\[
\sum_{n=-\infty}^{\infty}z^n\sum_{j=0}^{m-1}\frac{\omega^j}{1-\omega^jaq^n}.
\]
Now,
\begin{align*}
\sum_{j=0}^{m-1}\frac{\omega^j}{1-\omega^jaq^n}
&=-\sum_{j=0}^{m-1}\frac{1}{q^n}\frac{d}{da}\ln(1-\omega^jaq^n)\\
&=-\frac{1}{q^n}\frac{d }{da}\ln\left(\prod_{j=0}^{m-1}(1-\omega^jaq^n)\right)\\
&=-\frac{1}{q^n}\frac{d }{da}\ln\left(1-a^mq^{mn}\right)\\
&=\frac{ma^{m-1}q^{(m-1)n}}{1-a^mq^{mn}}.
\end{align*}
Hence the right side is
\[
ma^{m-1}\sum_{n=-\infty}^{\infty}\frac{(zq^{m-1})^n}{1-a^m(q^{m})^n},
\]
and \eqref{1psi1cReq2} follows  for $|q|<|z|<1$ after one further application of \eqref{prop1eq1} (with $q$ replaced with $q^m$, $a$ replaced with $a^m$, and $z$ replaced with $z q^{m-1}$). Once again,  the full result follows by analytic continuation.
\end{proof}

Note that if  $z$ is replaced with $z/a$ in the $m=2$ cases of \eqref{1psi1cReq1}  and  \eqref{1psi1cReq2}, then after some elementary $q$-product manipulations the following pair of identities appear:
 \begin{multline}\label{1psi1c1eq1}
 \left(z,\frac{q}{z},-a,-\frac{q}{a},q,q;q\right)_{\infty}
 +
 \left(-z,-\frac{q}{z},a,\frac{q}{a},q,q;q\right)_{\infty}\\
 =
 2\left(\frac{zq}{a},\frac{aq}{z},az,\frac{q^2}{az},q^2,q^2;q^2\right)_{\infty},
 \end{multline}
 \begin{multline}\label{1psi1c1eq2}
 \left(z,\frac{q}{z},-a,-\frac{q}{a},q,q;q\right)_{\infty}
 -
 \left(-z,-\frac{q}{z},a,\frac{q}{a},q,q;q\right)_{\infty}\\
 =
 2a
 \left(\frac{z}{a},\frac{aq^2}{z},azq,\frac{q}{az},q^2,q^2;q^2\right)_{\infty},
 \end{multline}
These are the two identities in \textbf{Entry 29} in chapter 16 of Ramanujan's second notebook (see \cite[p. 45]{B91}), so that in a sense \eqref{1psi1cReq1}  and  \eqref{1psi1cReq2} may be regarded as extensions of Ramanujan's two identities.

\section{Vanishing Coefficients}

 As remarked above, the right side of \eqref{midisszram}  provides an $p$-dissection of the left side in powers of $z$, in that each term in the sum on the right side contains powers of $z$ lying in a single arithmetic progression modulo $p$. For the purposes of giving simple proofs of some of the results stated in the introduction, the parameter $a$ is now specialized so that the coefficients of all powers of $z$ in one arithmetic progression vanish.

\begin{corollary}\label{c1}
Let $s$ be an integer, $0\leq s \leq p-1$. If the sequence $\{c_n(q)\}_{n=-\infty}^{\infty}$ is defined by
\begin{equation}\label{c1eq}
\frac{(q,q,q^{p-s}z^{1-p},q^{s+1 - p} z^{p-1};q)_{\infty}}
 {(q^{p - s} z^{-p},q^{s+1 - p} z^p,z,q/z;q)_{\infty}}
 = \sum_{n=-\infty}^{\infty}c_n(q)z^n.
\end{equation}
Then $c_{pn+s}(q)=0$.
\end{corollary}
\begin{proof}
Set $a = q^{p-s} z^{-p}$ in \eqref{midisszram}, so that the left side of \eqref{midisszram} becomes the left side of \eqref{c1eq}. The $j$-th term, for $0\leq j \leq p-1$, on the right side of \eqref{midisszram} becomes
\[
z^j
 \frac{(q^p,q^p,q^{p-s+j},q^{s-j};q^p)_{\infty}}
 {(q^{p-s+j}z^{-p},q^{s-j}z^p,z^p,q^p/z^p;q^p)_{\infty}},
\]
which vanishes when $j=s$, giving the result.
\end{proof}

There are a number of results in the literature about coefficients in the series expansion of various infinite $q$-products vanishing in certain arithmetic progressions, and also some results giving the $m$-dissection of various infinite products for some small positive integers $m>1$ (more details of some of these results are given below). These results have been proved in a number of ways, often involving various $q$-series identities and manipulation of infinite series, and in some cases using expansions of Hardy-Ramanujan-Rademacher type developed from the infinite products. In the present paper it is shown that many of these follow automatically from various forms of \eqref{midisszram} (such as \eqref{c1eq}), upon simply specializing some of the parameters (without any need for further manipulation of series or appeal to other $q$-series identities).

We first review some of the history of the results on vanishing coefficients (previous results on $m$-dissections will be described later, when discussing the results in the present paper on the topic).

In \cite{RS78},
Richmond and Szekeres showed that if the sequence $\{c_m\}$ is defined by
\begin{equation}\label{mod81}
F(q):=\frac{(q^3,q^5;q^8)_{\infty}}{(q,q^7;q^8)_{\infty}}=:\sum_{m=0}^{\infty}c_m q^m,
\end{equation}
then $c_{4n+3}$ is always zero.
Likewise, they showed that $d_{4n+2}$ is zero for all $n$, if the sequence $\{d_m\}$ is defined by
\begin{equation*}
\frac{1}{F(q)}=:\sum_{m=0}^{\infty}d_m q^m.
\end{equation*}
Their  results were derived  from Hardy-Ramanujan-Rademacher expansions they developed of the infinite products. They also conjectured that if
\begin{equation*}
G(q):=\frac{(q^5,q^7;q^{12})_{\infty}}{(q,q^{11};q^{12})_{\infty}}=:\sum_{m=0}^{\infty}a_m q^m,
\end{equation*}
then $a_{6n+5}$ is always zero, and if
\begin{equation*}
\frac{1}{G(q)}=:\sum_{m=0}^{\infty}b_m q^m,
\end{equation*}
then $b_{6n+3}$ is always zero.

In \cite{AB79}, Andrews and Bressoud proved the following theorem, from which the results of Richmond and Szekeres followed as special cases.
\begin{theorem}\label{t1}
If $1\leq r <k$ are relatively prime integers of opposite parity and
\begin{equation}\label{eq1}
\frac{(q^r,q^{2k-r};q^{2k})_{\infty}}{(q^{k-r},q^{k+r};q^{2k})_{\infty}}=:\sum_{n=0}^{\infty} \phi_nq^n,
\end{equation}
then $\phi_{kn+r(k-r+1)/2}$ is always zero.
\end{theorem}
Andrews and Bressoud used Ramanujan's $_1\psi_1$ summation formula \eqref{ram1}
to derive the result in Theorem \ref{t1},
after replacing $q$ with $q^{2k}$, specializing $a$, $b$ and $z$ and employing some $q$-series manipulations.
The  two results proved by Richmond and Szekeres, and the two results conjectured by them, follow from the
cases $(k,r)=(4,3),\,(4,1),\, (6,5)$ and $(6,1)$, respectively.

Alladi and Gordon \cite{AG94} prove a  generalization of Theorem \ref{t1}.
\begin{theorem}\label{AGt1}
Let $1<m<k$ and let $(s, km)=1$ with $1\leq s < mk$. Let $r^{*}=(k-1)s$ and $r\equiv r^{*} (\mod m k)$, with $1\leq r<mk$. \\
Put $r'=\lceil \frac{r^{*}}{mk}\rceil (\mod k)$ with $1 \leq r'<k$. Write
\[
\frac{(q^{r},q^{mk-r};q^{mk})_{\infty}}{(q^{s},q^{mk-s};q^{mk})_{\infty}}
=\sum_{n=0}^{\infty}a_n q^n.
\]
Then $a_n=0$ for $n \equiv rr'(\mod k)$.
\end{theorem}

Alladi and Gordon \cite{AG94}  also prove a companion result to that in their Theorem \ref{AGt1}.
\begin{theorem}\label{AGt2}
Let $m$, $k$, $s$, $r^{*}$, $r$, $r'$ be as in Theorem \ref{AGt2} with $k$ odd. Write
\[
\frac{(q^{r},q^{mk-r};q^{mk})_{\infty}}{(-q^{s},-q^{mk-s};q^{mk})_{\infty}}
=\sum_{n=0}^{\infty}a'_n q^n.
\]
Then $a'_n=0$ for $n \equiv rr'(\mod k)$.
\end{theorem}

In \cite{McL15}, the present author proved the following results.

\begin{theorem}\label{t2n}
Let $k>1$, $m>1$ be positive integers. Let $r=sm+t$, for some integers $s$ and $t$, where $0\leq s<k$, $1\leq t <m$ and $r$ and $k$ are relatively prime. Let
\begin{equation}\label{eq2n}
\frac{(q^{r-tk},q^{mk-(r-tk)};q^{mk})_{\infty}}{(q^{r},q^{mk-r};q^{mk})_{\infty}}=:\sum_{n=0}^{\infty} c_nq^n,
\end{equation}
then $c_{kn-rs}$ is always zero.
\end{theorem}

\begin{theorem}\label{t3n}
Let $k>1$, $m>1$ be positive integers, with $k$ odd. Let $r=sm+t$, for some integers $s$ and $t$, where $0\leq s<k$, $1\leq t <m$ and $r$ and $k$ are relatively prime. Let
\begin{equation}\label{eq3n}
\frac{(q^{r-tk},q^{mk-(r-tk)};q^{mk})_{\infty}}{(-q^{r},-q^{mk-r};q^{mk})_{\infty}}=:\sum_{n=0}^{\infty} d_nq^n,
\end{equation}
then $d_{kn-rs}$ is always zero.
\end{theorem}
The proofs of these theorems in that paper also employed the Ramanujan's $_1\psi_1$ summation formula
\eqref{ram1}, but in the present paper it is used in a different way to give an $m$-dissection of a special case of the product on the right side of \eqref{ram1}, from which these results follow automatically.

Theorems \ref{t2n} and \ref{t3n} now follow in a straightforward manner from Corollary \ref{c1}.

\begin{corollary}\label{c2}
Theorems \ref{t2n} and \ref{t3n} are true.
\end{corollary}
\begin{proof}
For Theorem \ref{t2n}, replace $q$ with $q^{mk}$ and set $p=k$ in Corollary \ref{c1}. Set $z=q^{mk-r}$, where $r$ is as in Theorem \ref{t2n} ($1 \leq r <mk$, $\gcd(r,k)=1$, so that $ r = sm + t$,
for some integers $s$ and $t$, where $0 \leq  s < k$, $1 \leq t < m$).
Let this particular value of $s$ be the $s$ in Corollary \ref{c1}.
Then
\begin{multline}\label{c2eq1}
\frac{(q^{km},q^{km},q^{r-k(r-ms)},q^{mk-(r-k(r-ms))};q^{km})_{\infty}}
{(q^{k(r-ms)},q^{km-k(r-ms)},q^{r},q^{km-r};q^{km})_{\infty}}
=\\
\sum_{j=0}^{k-1}
q^{j (k m-r)}
\frac{\left(
q^{k^2 m},
q^{k^2 m},
q^{k m(j+k-s)},
q^{k m (s-j)};
q^{k^2  m}
\right)_{\infty}}
{\left(
q^{k (k m-r)},
q^{k r},
q^{k (j m-s m+r)},
q^{k (-j m+k m+s m-r)};
q^{k^2 m}
\right)_{\infty }}.
\end{multline}
After multiplying both sides by the reciprocal of
\[
\frac{(q^{km},q^{km};q^{km})_{\infty}}
{(q^{k(r-ms)},q^{km-k(r-ms)};q^{km})_{\infty}}
\]
and noting that $r-ms=t$, the left side of \eqref{c2eq1} becomes the left side of \eqref{eq2n}. In the new right side of \eqref{c2eq1}, the $q$-products are such that all powers of $q$ are multiples of $k$, and since $r$ is invertible modulo $k$ (recall $\gcd(r,k)=1$),  each term of the sum contains all powers of $q$ in a single arithmetic progression modulo $k$. Theorem \ref{t2n} follows, since the term corresponding to $j=s$ is identically zero.

The proof of Theorem \ref{t3n} is similar, the only real difference being to set $z=-q^{mk-r}$.
\end{proof}

\section{$m$-Dissections}
In none of the papers referenced above did the authors give the relevant dissection of the infinite products, i.e. closed formulae for the arithmetic progressions in which the coefficients do not vanish. These dissections may be found elsewhere, for example the 4-dissection of the product \eqref{mod81} was given by Hirschhorn in \cite{H02} and Chen and Huang in \cite{CH05}, using different methods (and both methods differ from the method in the present paper). In those papers, producing the dissection required some effort, using various $q$-series manipulations and the Jacobi triple product identity.

As an illustration of the present method, where the dissections are produced automatically, their result is reproduced in the next corollary.
\begin{corollary}\label{c3}
The 4-dissection of the infinite product \eqref{mod81} is given by
\begin{multline}\label{c3eq1}
\frac{(q^3,q^5;q^8)_{\infty}}{(q,q^7;q^8)_{\infty}}=
\frac{(q^{12},q^{12},q^{20},q^{20};q^{32})_{\infty}}
{(q^{8},q^{16},q^{16},q^{24};q^{32})_{\infty}}\\
+\,q\,
\frac{(q^{4},q^{12},q^{20},q^{28};q^{32})_{\infty}}
{(q^{8},q^{8},q^{24},q^{24};q^{32})_{\infty}}
+\,q^2\,
\frac{(q^{4},q^{12},q^{20},q^{28};q^{32})_{\infty}}
{(q^{8},q^{16},q^{16},q^{24};q^{32})_{\infty}}
\end{multline}
\end{corollary}

\begin{proof}
In \eqref{c2eq1}, set $k=4$, $m=2$, $r=7$ and $s=3$. Multiply both sides by
\[
\frac{(q^{4},q^{4};q^{8})_{\infty}}
{(q^{8},q^{8};q^{8})_{\infty}}
\]
and the result follows upon rewriting the products $(q^{4};q^{8})_{\infty}$ and $(q^{8};q^{8})_{\infty}$ in terms of the base $q^{32}$.
\end{proof}

We next consider the limitations of the expansion provided by \eqref{midisszram} in producing $m$-dissections. In this identity, if $q$ is replaced with $q^t$ and then the substitutions $z=q^r$ and $a=q^s$ are made (some integers $r$ and $s$ satisfying $1 \leq r,s<t$), the identity
 \begin{equation}\label{midisszramt}
 \frac{(q^t,q^t,q^{r+s},q^{t-r-s};q^t)_{\infty}}
 {(q^s,q^{t-s},q^r,q^{t-r};q^t)_{\infty}}
 =
 \sum_{j=0}^{p-1}q^{jr}
 \frac{(q^{pt},q^{pt},q^{pr+s +jt},q^{(p-j)t-pr-s};q^{pt})_{\infty}}
 {(q^{jt+s},q^{(p-j)t-s},q^{pr},q^{(t-r)p};q^{pt})_{\infty}}
 \end{equation}
 is produced.
If it is desired that the right side represents a $p$-dissection of the left side, then it is necessary that \\
(a) all of the powers of $q$ in each of the infinite products on the right side must be multiples of $p$;\\
(b) The integer $r$ must satisfy $\gcd(r,p)=1$.

Hence $p|(sj+t),\, 0\leq j\leq p-1$, and hence $p|s$ and $p|t$, or $s=pu$ and $t=pv$ for some integers $u$ and $v$.

For example, setting $t=12$, $s=6$ and $r=1$ in \eqref{midisszramt} leads to
\begin{equation}\label{midisszramtex}
 \frac{(q^{12},q^{12},q^{7},q^{5};q^{12})_{\infty}}
 {(q^6,q^{6},q,q^{11};q^{12})_{\infty}}
 =
 \sum_{j=0}^{p-1}q^{j}
 \frac{(q^{12p},q^{12p},q^{12j+p+6},q^{12(p-j)-p-6};q^{12p})_{\infty}}
 {(q^{12j+6},q^{12(p-j)-6},q^{p},q^{11p};q^{12p})_{\infty}},
 \end{equation}
and letting $p=2,3$ and 6 immediately produces, in turn, the 2-, 3- and 6-dissections of
\begin{equation}\label{mod12pr}
 \frac{(q^{7},q^{5};q^{12})_{\infty}}
 {(q,q^{11};q^{12})_{\infty}}
\end{equation}
proven by Lin \cite{L13}, with no manipulation of infinite series or the use of other theta function identities being necessary (Lin also gave the 4-dissection of \eqref{mod12pr}, but the method described here does not provide that). We give the 3-dissection as an example.

\begin{corollary}\label{c4}(Lin, \cite[Theorem 3.2]{L13})
The 3-dissection of \eqref{mod12pr} is given by
\begin{multline}
\frac{(q^{7},q^{5};q^{12})_{\infty}}
 {(q,q^{11};q^{12})_{\infty}}
 =\frac{(q^{6},q^{30};q^{36})_{\infty}}
 {(q^{12},q^{12},q^{24},q^{24};q^{36})_{\infty}} \times\\
 \left(
 \frac{(q^{9},q^{18},q^{18},q^{27};q^{36})_{\infty}}
 {(q^{3},q^{33};q^{36})_{\infty}}
 +
 q\frac{(q^{6},q^{15},q^{21},q^{30};q^{36})_{\infty}}
 {(q^{3},q^{33};q^{36})_{\infty}}
+
q^2(q^{18},q^{18};q^{36})_{\infty}
 \right).
\end{multline}
\end{corollary}
\begin{proof}
Let $p=3$ in \eqref{midisszramtex}, and the result follows after a little algebra and some elementary $q$-product manipulations.
\end{proof}
Remarks: 1) Lin \cite{L13} also gave the 2-, 3-, 4- and 6-dissections of the reciprocal of the infinite product \eqref{mod12pr}, and the  2-, 3-  and 6-dissections may  be similarly produced by first
setting $t=12$, $s=6$ and $r=5$ in \eqref{midisszramt}, and then setting $p=2, 3$ and $6$ in turn in the resulting identity.

2) The reciprocal of the infinite product \eqref{mod12pr} is also equal to a special case of a quite general continued fraction recorded by Ramunajan in his second notebook (see \cite{R57}, \cite[Entry 12, page 24]{B91}),
\begin{multline}\label{cframentry12}
\frac{1}{1-ab}
\+
\frac{(a-bq)(b-aq)}{(1-ab)(1+q^2)}
\+
\frac{(a-bq^3)(b-aq^3)}{(1-ab)(1+q^4)}
\+
\cds \\
=
{\displaystyle
\frac{(a^2q^3,b^2q^3;q^4)_{\infty}}{(a^2q,b^2q;q^4)_{\infty}}},
\end{multline}
after replacing $q$ with $q^3$, then setting $a=q^2$ and $b=q$, and finally multiplying both sides by $1-q$. This special case was further investigated by Naika et al. in \cite{NDS08}, where  explicit evaluations, reciprocity theorems and integral representations were developed.

The formula \eqref{midisszramt} and the requirements that $p|t$ and $p|s$ also demonstrate some of the limitations of this method of producing $m$-dissections. For example, it cannot produce any dissections of products that arise from setting $t=5$, and hence does not lead to any results on vanishing coefficients (if they existed) in these cases. More generally, the formula \eqref{midisszramt} will not give dissections of any product that arises when $t=p$, $p$ a prime.


The formula \eqref{midisszramt} in some cases may be partially successful at recovering results proved elsewhere by other methods. For example, Hirschhorn and Roselin in \cite{HR10} consider Ramanujan's cubic continued fraction
\begin{align}\label{z3}
RC(q):=\frac{1}{1} \+
 \frac{q+q^{2}}{1}
\+
 \frac{q^{2}+q^{4}}{1}
\+
 \frac{q^{3}+q^{6}}{1}
\+ \cds =\frac{(q,q^5;q^{6})_{\infty}}{(q^{3},q^{3};q^{6})_{\infty}},
\end{align}
where the 2-, 3-, 4- and 6-dissections of $RC(q)$  and its reciprocal were given. The present method gives the 2-dissections of these, but not the 3-, 4- and 6-dissections.
\begin{corollary}\label{c5}
The 2-dissections of $RC(q)$ and $1/RC(q)$ are given by
\begin{align}\label{c5eq1}
RC(q)&=\frac{(q^4,q^{8};q^{12})_{\infty}}
{(q^6;q^{12})_{\infty} ^2}
\left(
\frac{(q^4,q^{8};q^{12})_{\infty}}
{(q^2,q^{10};q^{12})_{\infty}}
+q
\frac{(q^2,q^{10};q^{12})_{\infty}}
{(q^4,q^{8};q^{12})_{\infty}}
\right),\\
\frac{1}{RC(q)}&=
\frac{1}
{(q^6;q^{12})_{\infty} ^4}
\left(
(q^4,q^{8};q^{12})_{\infty}^2
-q (q^2,q^{10};q^{12})_{\infty}^2
\right).
\notag
\end{align}
\end{corollary}

\begin{proof}
For $RC(q)$, in \eqref{midisszramt} set $t=6$, $s = 2$, $r = 3$ and  $p = 2$, and for $1/RC(q)$ set $t=6$, $s = 2$, $r = 1$ and  $p = 2$. The results follow, as above, after some elementary algebra and $q$-series manipulations.
\end{proof}

Remarks: (1) These 2-dissections were represented differently by the authors in \cite{HR10}, where they were given in terms of infinite products of the form $(q^k;q^k)_{\infty}$.

(2) Recall that the Borwein cubic function $a(q)$ (see \cite{BB91}) is defined by
\[
a(q) = \sum_{m,n=-\infty}^{\infty}q^{m^2+mn+n^2}.
\]
We note in passing that it was shown in \cite{B98} that
\begin{equation}\label{aqeq}
a(q)-a(q^2) =
6q
\frac{(q,q^5,q^6,q^6;q^{6})_{\infty}}
{(q^2,q^{3},q^{3},q^{4};q^{6})_{\infty}},
\end{equation}
where the infinite product on the right side is the same as that on the left side of \eqref{midisszramt}, after  setting $t=6$, $s = 2$, $r = 3$, the same choices that were  made in deriving the 2-dissection of $RC(q)$.

\section{An identity that follows from Bailey's $_6\psi_6$ summation formula.}

The next identity, which we believe to be a new observation, is somewhat similar to that \eqref{midisszram}, but unfortunately does not appear to have any applications in terms of $m$-dissections or vanishing coefficients. However, we include it as we believe it deserves to be known, and it may have applications which the author does not presently see.

\begin{theorem}
If $|q|<1$ and $m$ is a positive integer, then
\begin{multline}\label{6psi6teq1}
{\displaystyle
\frac{\left(q,q,a,\frac{q}{a},\frac{b q}{d}, \frac{dq}{b}, \frac{aq}{b d},
\frac{b d q}{a};q\right)_{\infty }}
{\left(b,\frac{q}{b},d,\frac{q}{d},\frac{a}{b},\frac{bq}{a},\frac{a}{d},\frac{dq}{a};q\right)_{\infty }}}\\
   =
\sum _{r=0}^{m-1} q^r
\frac{
\left(q^m,q^m,a q^{2 r},\frac{q^{m-2 r}}{a},\frac{b q^m}{d},\frac{d q^m}{b},
\frac{a q^m}{b d},\frac{b dq^m}{a};q^m\right){}_{\infty }}
{\left(b q^r,\frac{q^{m-r}}{b},d q^r,\frac{q^{m-r}}{d},\frac{a q^r}{b},\frac{b q^{m-r}}{a},\frac{a q^r}{d},
\frac{dq^{m-r}}{a};q^m\right){}_{\infty }}.
\end{multline}
\end{theorem}

\begin{proof}
In Bailey's \cite{B36} $_{6}\psi_6$ identity (which holds for $|qa^2/(bcde)|<1$)
\begin{multline}\label{8baileyeq1}
\sum_{n=-\infty}^{\infty}\frac{(q\sqrt{a},-q\sqrt{a},b,c,d,e;q)_n}
{(\sqrt{a},-\sqrt{a},aq/b,aq/c,aq/d,aq/e;q)_n}\left(\frac{q a^2}{bcde}\right)^n\\
=\frac{ (aq,aq/bc,aq/bd,aq/be,aq/cd,aq/ce,aq/de,q,q/a;q)_{\infty} } {
(aq/b,aq/c,aq/d,aq/e,q/b,q/c,q/d,q/e,qa^2/bcde;q)_{\infty} },
\end{multline}
set $c=a/b$, $e=a/d$, and multiply both sides of the resulting identity by
\[
\frac{(1 - a)}{ (1 - b) (1 - d) (1 - a/b) (1 - a/d)}
\]
to get
\begin{multline}\label{8baileyeq2}
\sum_{n=-\infty}^{\infty}
\frac{(1 - a q^{2 n}) q^ n}{(1 - b q^n) (1 - d q^n) (1 - a q^n/b) (1 - a q^n/d)}\\
= {\displaystyle
\frac{\left(q,q,a,\frac{q}{a},\frac{b q}{d}, \frac{dq}{b}, \frac{aq}{b d},
\frac{b d q}{a};q\right)_{\infty }}
{\left(b,\frac{q}{b},d,\frac{q}{d},\frac{a}{b},\frac{bq}{a},\frac{a}{d},
\frac{d q}{a};q\right)_{\infty }}}.
\end{multline}
The result \eqref{6psi6teq1} now follows, after splitting the sum on the left of \eqref{8baileyeq2} into sums of terms in the same arithmetic progression modulo $m$ ($n\to mn+r$, $-\infty <n <\infty$, $0\leq r \leq m-1$) and using \eqref{8baileyeq2} on each of these sums (with the replacements $q\to q^m$, $a \to a q^{2r}$, $b\to b q^r$ and $d \to d q^r$).
\end{proof}

One might ask if is possible to specialize the parameters in \eqref{6psi6teq1} to derive an interesting zero sum. The obvious choice, setting $a=1$, does not lead to anything interesting, as the non-zero terms in the sum on the right cancel in pairs. However, another choice does lead to something less trivial.

\begin{corollary}
If $|q|<1$ and $m$ is  a positive integer, then
\begin{equation}
\sum_{r=0}^{m}
\frac{\left(ad q^{2r},\displaystyle{\frac{q^{m}}{adq^{2r}}};q^m\right)_{\infty}q^r}
{\left(aq^r,\displaystyle{\frac{q^{m}}{aq^r}},
a q^{r-1},\displaystyle{\frac{q^{m}}{aq^{r-1}}},
d q^r, \displaystyle{\frac{q^{m}}{dq^r}},
dq^{r+1},\displaystyle{\frac{q^{m}}{dq^{r+1}}};q^m\right)_{\infty} }
=0
\end{equation}
\end{corollary}

\begin{proof}
In \eqref{6psi6teq1}, set $b=dq$,  then replace $a$ with $ad$, and finally cancel any $q$-products that are independent of the summation variable $r$.
\end{proof}

\section{Miscellaneous Identities}

In this section we consider several other functions, each of which has a number of different series representations, one of which is in terms of Lambert series. The Lambert series representation in each case is such that when the series is ``sectioned'', so that when terms in each arithmetic progression modulo $m$ ($m>1$ a positive integer) are grouped together, each of the resulting new Lambert series is of the same form as the original, albeit with shifted parameters. This in turn permits each of the other series representations to be similarly expressed as a sum of $m$ series of the same type. To illustrate this, we first recall some other results  of the present author (see \cite{McL10}). Those results are collected and slightly reformulated to make easier to express the new results.

Recall that a \emph{WP-Bailey
pair} is a pair of sequences $(\alpha_{n}(a,k,q)$,
$\beta_{n}(a,k,q))$  satisfying $\alpha_{0}(a,k,q)$
$=\beta_{0}(a,k,q)=1$, and for $n>0$, {\allowdisplaybreaks
\begin{align}\label{WPpair}
\beta_{n}(a,k,q) &= \sum_{j=0}^{n}
\frac{(k/a;q)_{n-j}(k;q)_{n+j}}{(q;q)_{n-j}(aq;q)_{n+j}}\alpha_{j}(a,k,q).
\end{align}
}If $k=0$, then the pair of sequences $(\alpha_n(a,q),\beta_n(a,q))$ is called a \emph{Bailey pair with respect to $a$}.

\begin{theorem}\label{f123eq}
 Let $(\alpha_{n}(a,k,q)$,
$\beta_{n}(a,k,q))$ be a WP-Bailey pair and let $f_i(a,k,z,q)$, $i=1,2,3$ be defined as follows:
{\allowdisplaybreaks
  \begin{align}
  f_1(a,k,z,q)
&=
\sum_{n=1}^{\infty}\left(\frac{k q^n}{1-kq^n}+\frac{ q^n a/z}{1-q^n a/z}-
\frac{a q^n}{1-aq^n}- \frac{ q^n k/z}{1-q^n k/z}\right),\\
 f_2(a,k,z,q)
&=\sum_{n=1}^{\infty}
\frac{(1-k q^{2n})(z,k/a;q)_{n}}
 {(1-kq^n)(q k/z,q a;q)_{n}(1-q^n)}\left( \frac{q a}{ z }\right )^{n},\notag\\
f_3(a,k,z,q)
&=
\sum_{n=1}^{\infty} \frac{(q\sqrt{k},-q\sqrt{k},z;q)_{n}(q;q)_{n-1}}
{(\sqrt{k},-\sqrt{k}, q k,q k/z;q)_{n}}\left( \frac{q a}{ z }\right
)^{n} \beta_n(a,k,q) \notag\\
& \phantom{asdasdf}
- \sum_{n=1}^{\infty}\frac{(z;q)_{n}(q;q)_{n-1}}{(q a ,q
a/z;q)_n}\left (\frac{q a}{z}\right)^{n}\alpha_n(a,k,q).\notag
  \end{align}
  }
  Then for all quadruples of values $(a,k,z,q)$ for which all three functions are defined, we have
  \begin{equation}
  f_1(a,k,z,q)=f_2(a,k,z,q)=f_3(a,k,z,q).
  \end{equation}
\end{theorem}
The point here is that it follows easily from the definition of $f_1(a,k,z,q)$ (upon splitting the series into arithmetic progressions modulo $m$ (after the replacement $n\to mn-r$, $n\geq 1$, $0\leq r \leq m-1$), that if $m>1$ is an integer, then
\begin{equation}\label{f1eq}
  f_1(a,k,z,q)=\sum_{r=0}^{m-1}f_1(a q^{-r}, k q^{-r},z,q^m),
\end{equation}
and thus that \eqref{f1eq}  also holds if $f_1(a,k,z,q)$ is replaced with $f_2(a,k,z,q)$ or $f_3(a,k,z,q)$.

We ignore the choices for a WP-Bailey pair in $f_3(a,k,z,q)$, and instead use $f_2(a,k,z,q)$ to illustrate what has been said with the following somewhat unusual identity (which follows after the replacement $r \to m-r$).
\begin{corollary}
If $|qa/z|<1$, then for each positive integer $m$,
\begin{multline}\label{f2eq}
  \sum_{n=1}^{\infty}
\frac{(1-k q^{2n})(z,k/a;q)_{n}}
 {(1-kq^n)(a q,k q/z;q)_{n}(1-q^n)}\left( \frac{q a}{ z }\right )^{n}\\
 = \sum_{r=1}^{m}
 \sum _{n=1}^{\infty } \frac{
  \left(1-k q^{
   m (2n-1)+r}\right) \left(z,\frac{k}{a};q^m\right){}_n}
   {\left(1-k
   q^{m (n-1)+r}\right) \left(a q^{r},\frac{k
   q^{r}}{z};q^m\right){}_n\left(1-q^{m n}\right) }\left( \frac{q^{r} a}{ z }\right )^{n}.
\end{multline}
\end{corollary}
Upon letting $k\to 0$ in Theorem \ref{f123eq}, then the next corollary follows.

\begin{corollary}\label{cobapr}
Let $(\alpha_n(a,q), \beta_n(a,q))$ be a Bailey pair with respect to $a$, and let $g_i(a,z,q)$,  $i=1,2,3,4,$ be defined by
{\allowdisplaybreaks
\begin{align}\label{gieqs}
g_1(a,z,q)
&=\sum_{n=1}^{\infty}\left(\frac{a q^n/z}{1-a q^n/z}-\frac{a q^n}{1-aq^n}\right),\\
g_2(a,z,q)
&=\sum_{n=1}^{\infty} \frac{(z;q)_{n}}
{(q a;q)_{n}(1-q^n)}\left( \frac{q a}{ z }\right )^{n}\notag\\
g_3(a,z,q)
&=
\sum_{n=1}^{\infty} (z;q)_{n}(q;q)_{n-1}\left( \frac{q a}{ z }\right )^{n} \beta_n(a,q)\notag\\
& \phantom{asdasdf} - \sum_{n=1}^{\infty}\frac{(z;q)_{n}(q;q)_{n-1}}{(q a ,q
a/z;q)_n}\left (\frac{q a}{z}\right)^{n}\alpha_n(a,q), \notag\\
g_4(a,z,q)
&=-\sum_{n=1}^{\infty} \frac{(1-a q^{2n})(z;q)_{n}q^{n(n+1)/2}}
{(1-a q^n)(qa/z;q)_{n}(1-q^n)}\left( \frac{- a}{ z }\right )^{n}.\notag
\end{align} }
Then, for any triple of values for $a$, $z$ and  $q$ for which all of the series converge, it holds that
\begin{equation}\label{cobapreq}
  g_1(a,z,q)=g_2(a,z,q)=g_3(a,z,q)=g_4(a,z,q).
\end{equation}
\end{corollary}
Note that $g_4(a,z,q)$ is a special case of $g_3(a,z,q)$, derived by inserting the ``unit" Bailey pair
{\allowdisplaybreaks
\begin{align*}
\alpha_{n}(a,q)&=\frac{(q \sqrt{a}, -q
\sqrt{a},a;q)_n}{(\sqrt{a},-\sqrt{a},q;q)_n}\left(-1\right)^n q^{n(n-1)/2},\\
\beta_n(a,q)&=\begin{cases} 1&n=0, \notag\\
0, &n>1,
\end{cases}
\end{align*}
}
By a similar argument to that above, it easily follows that if $m$ is any positive integer, then for $1\leq i \leq 4$,
\begin{equation}\label{gieq}
  g_i(a,z,q)=\sum_{r=0}^{m-1}g_i(a q^{-r}, z,q^m).
\end{equation}
We defer giving a specific example until after stating the special case of Corollary \ref{cobapr} derived by letting $z\to \infty$.

\begin{corollary}\label{cobapr2}
Let $(\alpha_n(a,q), \beta_n(a,q))$ be a Bailey pair with respect to $a$, and let $h_i(a,q)$,  $i=1,\dots , 6$ be defined by
{\allowdisplaybreaks
\begin{align}\label{hieqs}
h_1(a,q)
&=\sum_{n=1}^{\infty}\frac{a q^n}{1-aq^n},\\
h_2(a,q)
&=-\sum_{n=1}^{\infty} \frac{q^{n(n+1)/2}\left(-a\right)^{n}}
{(q a;q)_{n}(1-q^n)}\notag\\
h_3(a,q)
&=
-\sum_{n=1}^{\infty} (q;q)_{n-1}\left( - a\right )^{n}q^{n(n+1)/2}
\beta_n(a,q) \notag\\
& \phantom{asdasdf}+ \sum_{n=1}^{\infty}\frac{(q;q)_{n-1}\left (-
a\right)^{n}q^{n(n+1)/2}}{(q a ;q)_n}\alpha_n(a,q), \notag\\
h_4(a,q)
&=\sum_{n=1}^{\infty} \frac{(1-aq^{2n})q^{n^2} a^{n}}
{(1-a q^{n})(1-q^n)},\notag \\
h_5(a,q)
&=\frac{-1}{(aq;q)_{\infty}}\sum_{n=1}^{\infty}\frac{n (-a)^n q^{n(n+1)/2}}{(q;q)_n},\notag\\
h_6(a,q)
&=(aq;q)_{\infty}\sum_{n=1}^{\infty}\frac{n a^n q^n}{(q;q)_n}. \notag
\end{align} }
Then, for any pair of values for $a$ and  $q$ for which all of the series converge, it holds that
\begin{equation}\label{cobapreq2}
  h_1(a,q)=h_2(a,q)=h_3(a,q)=h_4(a,q)=h_5(a,q)=h_6(a,q).
\end{equation}
\end{corollary}
Remark: The equivalence of $h_5(a,q)$ and $h_6(a,q)$ to $h_1(a,q),\dots ,h_4(a,q)$ is shown in \cite{McL10}.

Once again, by reasoning as above, it is easy to show that if $m$ is any positive integer, then for $1\leq i \leq 6$,
\begin{equation}\label{hieq}
  h_i(a,q)=\sum_{r=0}^{m-1}h_i(a q^{-r}, q^m).
\end{equation}
The following pair of somewhat curious identities follow from $h_5(a,q)$ and $h_6(a,q)$ (after making the replacement $r \to m-r$ on the left side).
\begin{corollary}
Let $m$ be a positive integer.\\
(i) If $|q|<1$, then
\begin{multline}\label{h5eq}
\frac{1}{(aq;q)_{\infty}}\sum_{n=1}^{\infty} \frac{n(-a)^n q^{n(n+1)/2}}{(q;q)_n}\\
= \sum_{r=1}^{m}\frac{1}{(aq^{r};q^m)_{\infty}}\sum_{n=1}^{\infty} \frac{n(-a)^n q^{mn(n-1)/2+nr}}{(q^m;q^m)_n}.
\end{multline}
If $|aq|<1$, then
\begin{equation}\label{h6eq}
(aq;q)_{\infty}\sum_{n=1}^{\infty} \frac{n a^n q^{n}}{(q;q)_n}
= \sum_{r=1}^{m}(aq^{r};q^m)_{\infty}\sum_{n=1}^{\infty} \frac{na^n q^{n r}}{(q^m;q^m)_n}.
\end{equation}
\end{corollary}

\section{Further applications - the Fine function $F(a,b;t)$}
The ideas in the present paper may be applied in a number of other situations. One of these is the collection of identities that exist involving the function
\begin{equation}\label{9fabtq}
F(a,b;t):=\sum_{n=0}^{\infty} \frac{(a q;q)_n}{(b q;q)_n}t^n,
\end{equation}
which was investigated in \cite{F88} by Fine. One important identity involving the function $F(a,b;t)$ that he found was the
\emph{Rogers-Fine identity},  discovered previously by Rogers \cite{R17}:
\begin{equation}\label{9rf1}
\sum_{n=0}^{\infty} \frac{(\alpha;q)_n}{(\beta;q)_n}\tau^n
=
\sum_{n=0}^{\infty} \frac{(1-\alpha \tau q^{2n})(\alpha, \alpha \tau q/\beta;q)_n\beta^n\tau^n q^{n^2-n}}{(\beta;q)_n(\tau;q)_{n+1}}.
\end{equation}
Note that the left side of \eqref{9rf1} is $F(\alpha/q, \beta/q;\tau)$.
Fine derived several other identities involving the function on the left side of \eqref{9rf1} (for simplicity in the present paper, results are stated in terms of this series, rather than $F(a,b;t)$), including
\begin{equation}\label{9rf2}
\sum_{n=0}^{\infty} \frac{(\alpha;q)_n}{(\beta;q)_n}\tau^n
=
\sum_{n=0}^{\infty} \frac{(\beta/\alpha;q)_n(-\alpha t)^n q^{n(n-1)/2}}{(\beta ;q)_n(t;q)_{n+1}}.
\end{equation}
From the perspective of the present paper, the special cases of these identities that are of interest arise from setting $\beta = \alpha q$ and dividing both sides in each case by $1-\alpha$. This leads to
\begin{align}\label{9rf3}
F_1(\alpha,\tau,q):=\sum_{n=0}^{\infty}\frac{\tau^n}{1-\alpha q^n}
&=\sum_{n=0}^{\infty}\frac{(1-\alpha \tau q^{2n})\alpha^n \tau^nq^{n^2}}
{(1-\alpha  q^{n})(1- \tau q^{n})}=:F_2(\alpha,\tau,q)\notag\\
&=\sum_{n=0}^{\infty} \frac{(q;q)_n(-\alpha t)^n q^{n(n-1)/2}}{(\alpha ;q)_{n+1}(t;q)_{n+1}}=:F_3(\alpha,\tau,q).
\end{align}
The point in stating these special cases is that, by similar reasoning to that employed earlier,
it follows that if $m$ is any positive integer, then
\begin{equation}\label{9rf4}
  F_i(\alpha,\tau,q)=\sum_{r=0}^{m-1}t^r F_i(\alpha q^r,\tau^m,q^m),\hspace{25pt} i=1,2,3.
\end{equation}
A similar result may be derived from any other identity for the function $F(a,b;t)$.

We leave it to the reader to possibly find some useful applications of these facts.

{\allowdisplaybreaks

}

\end{document}